\documentclass[a4paper]{amsart}

\usepackage{microtype}
\usepackage[ruled, vlined, linesnumbered, longend]{algorithm2e}

\usepackage{amssymb}

\usepackage{bbold}

\usepackage[hidelinks]{hyperref}
\usepackage{amsrefs}

\newcommand\br[1]{\left(#1\right)}
\newcommand\abs[1]{\left|#1\right|}
\newcommand\set[1]{\left\{#1\right\}}

\newcommand{\NN}{\mathbb{N}}
\newcommand{\ZZ}{\mathbb{Z}}
\newcommand{\QQ}{\mathbb{Q}}

\newtheorem{thm}{Theorem}[section]
\newtheorem{conjecture}[thm]{Conjecture}
\newtheorem{lem}[thm]{Lemma}

\newtheorem{remark}[thm]{Remark}

\begin{document}

\title[$x^2-a\, y^2 = 1$ and $z^2 - b\, x^2 = 1$ have at most one solution.]{On the unique solvability of the simultaneous Pell equations $x^2-a\, y^2 = 1$ and $z^2 - b\, x^2 = 1$.}

\author[T. Hilgart]{Tobias Hilgart}
\email{tobias.hilgart@plus.ac.at}

\author[V. Ziegler]{Volker Ziegler}
\email{volker.ziegler@plus.ac.at}



\begin{abstract}    
    We consider the simultaneous Pell equations
    \[
        x^2 - a\, y^2 = 1, \qquad z^2 - b\, x^2 = 1,
    \]
    where $a > b\geq 2$ are positive integers. We describe a procedure which, for any fixed $b$, either confirms that the simultaneous Pell equations have at most one solution in positive integers, or finds all exceptions for which we have proved that there are at most finitely many.
\end{abstract}

\maketitle

%
%
\section{Introduction}
\label{sec:intro}

    A natural question to ask when considering simultaneous Pell equations
    \[
        a\, x^2 - b\, z^2  = \delta_1, \qquad c\, y^2 - d\, x^2 = \delta_2
    \]
    is, how many solutions (in positive integers) there are. Ineffective results from Thue~\cite{thue1909} and Siegel~\cite{siegel1929} imply that for $\gcd(ab, \delta_1) = \gcd(cd, \delta_2) = 1$ and $b\delta_1 \neq c \delta_1$ there are only finitely many. 

    If neither $ab$ nor $bd$ is a perfect square, and $\gcd(abd, \delta) = 1$, then Yuan \cite{yuan2004} proved that if each equation is solvable, the simultaneous equations 
    \[
        a\, x^2 - b\, z^2 = \delta, \qquad c\, y^2 - d\, z^2 = \delta
    \]
    have a solution (even an infinite number of them) if and only if $ad$ is a perfect square. 

    Masser and Rickert \cite{massrick1996} proved that, given the flexibility to freely choose both $\delta_1$ and $\delta_2$, simultaneous Pell equations can be constructed with arbitrarily many solutions.

    That is, it makes sense to narrow the scope when considering simultaneous Pell equations; in this paper, we consider
    \begin{equation}\label{eq:simultPell}
        x^2 - a\, y^2 = 1, \qquad z^2 - b\, x^2 = 1.
    \end{equation}

    Bennett \cite{bennett1998} noted that an analogous version of his proof shows that the above equations have at most three solutions. Cipu and Mignotte \cite{cipumign2007} improved this bound to two solutions, and reformulated a conjecture of Yuan~\cite{yuan2004} on the number of solutions, 
    \begin{conjecture}
        If the system of simultaneous Pell equations
        \[
            a\, x^2 - b\, y^2 = 1, \qquad c\, y^2 - d\, z^2 = 1
        \]
        has at least two solutions in positive integers, then their coefficients are given by
        \[
            (a, b, c, d) = (1, m^2-1, 1, n(l, m)^2 - 1), \qquad (a, a-1, b(l,a), b(l,a) -1)
        \]
        or equivalent forms thereof, where
        \[
            n(l, m) = \cfrac{ \br{ m + \sqrt{m^2 - 1} }^{2l} - \br{ m - \sqrt{m^2-1} }^{2l} }{4 \sqrt{m^2 - 1}}
        \]
        and
        \[
            4b(l, a) - 1 = \cfrac{ \br{ \sqrt{a} + \sqrt{a-1} }^l - \br{ \sqrt{a} - \sqrt{a-1} }^l }{2 \sqrt{a-1}}, \quad l \equiv 3 \mod 4.
        \]
    \end{conjecture}

    The conjecture states for the system \eqref{eq:simultPell} that there is at most one solution in positive integers, except maybe some special cases where $a$ and $b$ are almost squares.

    For several classes of coefficients, the above conjecture was confirmed. Yuan himself \cite{yuan2004specific} proved it for $b$ of the form $4m(m+1)$, and Cipu \cite{cipu2007} for $4m^2-1$. Zhang et al. \cite{zhang2015} completely solved the case where $a$ is prime and $b = 24$, not only confirming the conjecture in this case, but finding that there are only the solutions $(99,70,485)$ and $(10,3,49)$ when $a = 2$ and $a = 11$, respectively.

    In the spirit of \cite{zhang2015}, we consider $b$ to be fixed -- but now arbitrary. Without any further restrictions, we prove that the conjecture (that there is at most one solution) is eventually true, in the sense that when $a$ is sufficiently large with respect to $b$. We formulate this in the following 
    \begin{thm}\label{thm:main}
        Let $a > b \geq 2$ be two positive integers, where $b$ is fixed. Then there exists an effectively computable constant $a_0$, dependent only on $b$, such that for $a \geq a_0$ the simultaneous Pell equations
        \begin{equation}\label{eq:pell}
            x^2 - a\, y^2 = 1, \qquad z^2 - b\, x^2 = 1
        \end{equation}
        have at most one solution $(x,y,z)$ in positive integers.
    \end{thm}    

    If we follow the proof of the above theorem and depart from it at the right time, we can use the theory of continued fractions and a well-known reduction method (or variant thereof) based on it to obtain an executable procedure that can confirm the above conjecture for any (reasonably large) $b$. The algorithm is described and analysed in Section~\ref{sec:procedure}. Our implementation in Sagemath~\cite{sagemath} has yielded the following results.
    \begin{thm}\label{thm:brun}
        For $1 \leq b \leq 10000$, there are no two solutions to the simultaneous Pell equations
        \[
            x^2 - a\, y^2 = 1, \qquad z^2 - b\, x^2 = 1,
        \]
        for any $a \geq b$.        
    \end{thm}

%
%
\section{Auxiliary Results}
\label{sec:lemmas}

    In our proof of Theorem~\ref{thm:main}, we use the fundamental solutions of both Pell equations to derive linear forms in logarithms and use lower bounds on them to obtain bounds on the coefficients of the Pell equations. We use the following results due to Matveev~\cite{matveev2000} and Laurent~\cite{laurent2008}, respectively, depending on whether we have linear forms in three or two logarithms.
    
    \begin{thm}[\cite{matveev2000}*{Corollary~2.3; real case and $n=3$}]\label{thm:matveev}
        Let $\alpha_1$, $\alpha_2$ and $\alpha_3$ be three distinct positive real algebraic numbers whose logarithms are $\ZZ$-linearly independent, let $b_1$, $b_2$ and $b_3$ be integers with $b_3 \neq 0$, and put
        \[
            \Lambda = b_1 \log\alpha_1 + b_2 \log\alpha_2 + b_3 \log\alpha_3.
        \]
        
        Let
        \[
            D = [\QQ(\alpha_1, \alpha_2, \alpha_3) : \QQ],
        \]
        and let $A_1$, $A_2$ and $A_3$ be positive real numbers such that
        \[
            A_j \geq \max\set{ D\, h(\alpha_j), \abs{\log\alpha_j} } \quad (1 \leq j \leq 3),
        \]
        where $h$ is the absolute Weil height.

        Choose any $B$ satisfying
        \[
            B \geq \max\set{ \abs{b_j} \frac{A_j}{A_1}\, : \, 1 \leq j \leq 3 },
        \]
        and define
        \[
            c_{\text{Mat}} = \frac{45 \cdot 16^5}{4}\, e^4\, \br{ 26.25 + \log\br{ D^2 \log(eD) } } ,
        \]
        then
        \[
            \log\abs{\Lambda} > -c_{\text{mat}} D^2 A_1 A_2 A_3 \log\br{ 1.5eD \log(1.5eD) B }
        \]
    \end{thm}

    \begin{thm}[\cite{laurent2008}*{Corollary~2} ]\label{thm:laurent}
        Let $\alpha_1, \alpha_2 \geq 1$ be two multiplicatively independent real numbers, and let $b_1$ and $b_2$ be positive integers. Put $D = [\QQ(\alpha_1, \alpha_2) : \QQ]$ and set
        \[
            b' = \frac{b_1}{D \log A_2} + \frac{b_2}{D \log A_1},
        \]
        where $A_1$ and $A_2$ are real numbers greater than $1$ such that
        \[
            \log A_i \geq \max\set{ h(\alpha_i), \log\alpha_i/D, 1/D } \quad (i=1,2).
        \]
        
        Then
        \[
            \log\abs{ b_1 \log\alpha_1 - b_2 \log\alpha_2 } \geq -17.9 D^4 \br{ \max\set{ \log b' + 0.38, 30/D, 1 } }^2 \log A_1 \log A_2.
        \]
    \end{thm}

    Switching to the implementation in Section~\ref{sec:procedure}, where before we had a linear form in fixed but arbitrary logarithms, there we know them explicitly. We think of it as a rational approximation to an explicitly known number and simply use the following result involving continued fractions.

    \begin{lem}\label{lem:continuedfraction}
        Assume that $\mu$ is real and irrational and has the continued fraction expansion $\mu = [a_0; a_1, a_2, \ldots]$. Let $\ell$ be a positive integer, set $A=\max_{j=1, \ldots, \ell}\set{a_j}$, and let $\frac{p_\ell}{q_\ell}$ be the $\ell$-th convergent to $\mu$, then
        \[
            \frac{1}{\br{2+A} q_\ell} < \abs{p - q \mu } 
        \]
        for any rational fraction $\frac{p}{q}$ with  $q \leq q_\ell$.
    \end{lem}
    \begin{proof}
        See, for instance, page~$47$ of \cite{baker1984}.
    \end{proof}

    Finally, to reduce the bounds obtained, we use a variant of a reduction method of Baker and Davenport~\cite{bakdav1969}, which can be found in \cite{odjzieg2023}.

    \begin{thm}[\cite{odjzieg2023}*{Lemma~6}]\label{thm:reduction}
        Consider a Diophantine inequality of the form
        \[
            \abs{ n \mu + \tau - x } < c_1 \exp\br{ -c_2 n },
        \]
        where $n \in \NN$, $x \in \ZZ$, $c_1$ and $c_2$ are positive constants, and $\mu$ and $\tau$ are real numbers. Suppose $n < N$ and $\kappa > 1$ such that there exists a convergent $p/q$ to $\mu$, where
        \[
            \set{q\,\mu} < \frac{1}{2\kappa N} \qquad \text{and} \qquad \set{q\,\tau} > \frac{1}{\kappa},
        \]
        where $\set{\cdot}$ denotes the distance to the nearest integer. Then we have
        \[
            n \leq \frac{ \log\br{2\kappa q c_1} }{c_2}.
        \]
    \end{thm}

%
%
\section{Proof of Theorem~\ref{thm:main}}
\label{sec:mainproof}

    The pair $(z_1, x_1)$ solving the equation $z^2 - b\, x^2 = 1$ and minimising $z$ is given by some convergent of $\sqrt{b}$ and is called the fundamental solution of the Pell equation. It has the well-known property that all solutions to the Pell equation can be constructed via
    \[
        (z_i + \sqrt{b} x_i) = (z_1 + \sqrt{b} x_1)^k, \qquad k=1, 2 \ldots.
    \]

    This is why the unit $z_1 + \sqrt{b}\, x_1$ is sometimes also called the fundamental solution. We call it $\epsilon$ and take the unit such that $\epsilon > 1$ by changing the sign and/or taking the inverse if necessary. Similarly, we denote the fundamental solution greater than $1$ of $x^2 - a\, y^2 = 1$ by $\delta_a$. In particular, this means that
    \begin{equation}\label{eq:fu_obviousbounds}
        \delta_a \geq 2\sqrt{a}, \qquad \epsilon \geq 2\sqrt{b}.
    \end{equation}

    Let $(x,y,z)$ be a solution to Equation~\eqref{eq:pell} in positive integers. Since
    \[
        x^2 - a\, y^2 = \br{ x + \sqrt{a}\, y}\br{x-\sqrt{a}\, y} = 1,
    \]
    there is a positive integer $m \geq 1$ such that
    \[
        x + \sqrt{a}\, y = \delta_a^m, \qquad x - \sqrt{a}\, y = \delta_a^{-m},
    \]
    and eliminating $y$ by summing the two equations gives
    \[
        x = \frac{ \delta_a^m + \delta_a^{-m} }{2}.
    \]

    We do the same thing for the second Pell equation and deduce that there exists a positive integer $n \geq 1$ such that
    \begin{equation}\label{eq:xbyepsdelta}
        x =  \frac{ \delta_a^m + \delta_a^{-m} }{2} = \frac{ \epsilon^n - \epsilon^{-n} }{ 2\sqrt{b} },
    \end{equation}
    and so we have
    \begin{equation}\label{eq:epsdelta}
        \frac{\epsilon^n}{\sqrt{b}} - \delta_a^m = \frac{\epsilon^{-n}}{\sqrt{b}} + \delta_a^{-m} > 0.
    \end{equation}

    But the right-hand side of the equation is small with respect to $m$ and $n$, so the two terms on the left-hand side must be asymptotically equal with respect to $m$ and $n$. We can specify this statement as
    \begin{equation}\label{eq:epsdelta_bounds}
        \delta_a^m < \frac{\epsilon^n}{\sqrt{b}}  < c_1 \, \delta_a^m,
    \end{equation}
    where $c_1$ is an upper bound to the expression
    \[
        1 + \frac{1}{ \delta_a^{2m} } + \frac{1}{ \delta_a^m \epsilon^n \sqrt{b} };
    \]
    since $\delta_a \geq 2\sqrt{a} > 2 \sqrt{b}$ and $m,n \geq 1$, we can take
    \[
        c_1 = 1 + \frac{1}{4b} + \frac{1}{2b\epsilon}.
    \]

    We then divide Equation~\eqref{eq:epsdelta} by $\delta_a^m$ and use the first inequality in Equation~\eqref{eq:epsdelta_bounds}, more precisely $\epsilon^{-n} < \delta_a^{-m} / \sqrt{b}$ to get
    \[
        0 < \frac{\epsilon^n}{\sqrt{b}} \delta_a^{-m} - 1 < \br{1 + \frac{1}{b}} \, \delta_a^{-2m} ,
    \]
    which, since $0 < \log\xi < \xi-1$ for $\xi > 1$, implies
    \begin{equation}\label{eq:linearform}
        0  <  -m \log\delta_a + n \log\epsilon - \log\sqrt{b}  <  \br{1 + \frac{1}{b}}  \delta_a^{-2m}.
    \end{equation}

    We then apply Theorem~\ref{thm:matveev} on this non-zero linear form, where we set
    \[
        D \leq 4, \quad A_1 = 2 \log\delta_a, \quad A_2 = 2 \log\epsilon, \quad A_3 = 4 \log\sqrt{b}.
    \]
    
    Using the second inequality in Equation~\eqref{eq:epsdelta_bounds}, we get
    \[
        n \frac{\log\epsilon}{\log\delta_a} < m + \frac{\log(c_1\sqrt{b})}{\log\delta_a},
    \]
    and with Equation~\eqref{eq:fu_obviousbounds}, we can set $B$ in Theorem~\ref{thm:matveev} as
    \[
        m + \frac{\log(c_1\sqrt{b})}{\log\delta_a} < m + \frac{\log(c_1\sqrt{b})}{\log2\sqrt{b}} < m + 0.52 = B.
    \]

    Plugging everything into Theorem~\ref{thm:matveev}, we compare its lower bound for the linear form with the upper bound from Equation~\eqref{eq:linearform}, which gives
    \[
        - c_{\text{Mat}} 16^2 \log\delta_a \log\epsilon \log\sqrt{b} \; \log\br{ 6e\log(6e)\, (m+0.52)} < -2m \log\delta_a  + \log(1+1/b). 
    \]
    
    Again using Equation~\eqref{eq:fu_obviousbounds} and the worst case $b=2$, we can bound $\frac{\log(1+1/b)}{\log\delta_a}$ by $\frac{\log(1+1/b)}{\log(2\sqrt{b})} < 0.39$. So, if we divide the above inequality by $\log\delta_a$, we get
    \begin{equation}\label{eq:matveev}
        - c_{\text{Mat}} 16^2 \log\epsilon \log\sqrt{b} \; \log\br{ 6e\log(6e)\, (m+0.52)} < -2m + 0.39,
    \end{equation} 
    and comparing both sides we find that there exists an effectively computable constant $c_m$ such that $m < c_m$.

    These arguments hold for every solution to \eqref{eq:pell}, whenever there is at least one solution. From now on, we assume that there are two solutions $(x_1, y_1, z_1)$ and $(x_2, y_2, z_2)$ with the corresponding exponents $(m_1, n_1)$ and $(m_2, n_2)$ that satisfy $n_1 < n_2$. Equation~\ref{eq:xbyepsdelta} and $\delta_a, \epsilon > 1$ imply that then $m_1 < m_2$ as well.

    Again we divide Equation~\eqref{eq:epsdelta} by $\delta_a^{m_1}$ and $\delta_a^{m_2}$, respectively, but this time we use the second inequality from Equation~\eqref{eq:epsdelta_bounds}, i.e. $\delta_a^{-m} < c_1 \sqrt{b}\, \epsilon^{-n}$, to get upper bounds in $\epsilon^{-2n_1}$ and $\epsilon^{-2n_2}$:
    \begin{equation}\label{eq:linearformbyeps}
        \begin{split}
            0  &<  -m_1 \log\delta_a  +  n_1 \log\epsilon  -  \log\sqrt{b}  <  \br{c_1^2 b + c_1} \epsilon^{-2n_1} \\
            0  &<  -m_2 \log\delta_a  +  n_2 \log\epsilon -  \log\sqrt{b}  <  \br{c_1^2 b + c_1} \epsilon^{-2n_2}
        \end{split}
    \end{equation}

    We take the first linear form $m_2$-times and subtract from it the $m_1$-th multiple of the second linear form. In this way, we arrive at the linear form
    \begin{equation}\label{eq:linearform_noDelta}
        \abs{ \br{ m_2 n_1 - m_1 n_2 } \log\epsilon - \br{ m_2 - m_1 } \log\sqrt{b} } < c_m\br{c_1^2 b + c_1} \epsilon^{-2n_1},
    \end{equation}
    from which we have eliminated $\log\delta_a$. We know that this linear form still does not vanish, since the logarithms are linearly independent and at least $m_2 - m_1$ is non-zero. With only two logarithms, we can use the often numerically better bounds of Laurent, Theorem~\ref{thm:laurent}; we set
    \[
        D=2, \quad \log A_1 = \frac{\log\epsilon}{2}, \quad \log A_2 = \log\sqrt{b}, \quad b' = \frac{\abs{m_2n_1 - m_1n_2}}{2 \log A_2} + \frac{\abs{m_2-m_1}}{2 \log A_1}.
    \]

    Using the fact that all $m_i$ and $n_i$ are positive and $m_1, m_2 < c_m$, we can bound $b'$ in terms of $n_2$. In addition, we can use Equtaion~\eqref{eq:epsdelta_bounds} twice in order to estimate $n_2$ in terms of $n_1$,
    \begin{align}\label{eq:n2_bound}
        n_2 &< m_2 \frac{\log \delta_a}{\log\epsilon} + \frac{ \log\br{c_1 \sqrt{b}} }{\log\epsilon} \nonumber\\
            &< c_m \frac{\log\delta_a}{\log\epsilon} + \frac{\log\br{c_1\sqrt{b}}}{\log\epsilon} \nonumber\\
            &< c_m \br{ n_1 + \frac{\log\sqrt{b}}{\log\epsilon} } + \frac{\log\br{c_1 \sqrt{b}}}{\log\epsilon},
    \end{align}
    and can thus effectively bound $b'$ by $n_1$. We can even make this explicit by writing
    \[
        b' < \frac{ c_m^2 ( n_1 + \log\sqrt{b} / \log\epsilon ) + c_m \log(c_1\sqrt{b})/\log\epsilon }{2 \log\sqrt{b}} + \frac{c_m}{\log\epsilon}.
    \]

    Applying Theorem~\ref{thm:laurent} to the linear form in Equation~\eqref{eq:linearform_noDelta} and comparing it to its upper bound gives
    \[
        17.9\cdot 2^3 \br{ \log b' + 0.38 }^2 \log\epsilon \log\sqrt{b} > 2n_1 \log\epsilon - \log\br{ c_m\br{c_1^2 b + c_1} },
    \]
    which in turn implies $n_1 < c_{n_1}$ for an effective constant $c_{n_1}$. It then follows from Equation~\eqref{eq:epsdelta_bounds} that
    \[
        \log\delta_a < c_{n_1} \log\epsilon - \log\sqrt{b} < \log(2\sqrt{a_0}),
    \]
    for an effective constant $a_0$. Finally, by the first inequality in Equation~\eqref{eq:fu_obviousbounds}, this implies that $a < a_0$, which completes the proof of Theorem~\ref{thm:main}.

%
%
\section{Implementation}
\label{sec:procedure}

The proof of Theorem~\ref{thm:main} is effective in the sense that it describes a procedure that we can (theoretically) follow for any explicit b to eventually obtain an explicit $a_0$. We can then check for the remaining "small" $a$ whether any of the simultaneous Pell equations have more than one solution. 

However, the problem is that we use lower bounds for linear forms in logarithms twice, to get the bounds $c_m$ and $c_{n_1}$, respectively. For example, for $b=24$ and compared with \cite{zhang2015}, we know that the fundamental solution of $z^2 - 24\,x^2 = 1$ is $5 + \sqrt{24}$, which we can use for $\epsilon$. If we calculate the two constants, we would get $c_m \approx 3.345 \times 10^{14}$ and $c_{n_1} = 720888$. This would lead to an absolutely staggering bound of $a_0 \approx 10^{ 3.3 \times 10^6 }$ for $a$.

So instead we depart from the theoretical proof after obtaining $c_m$. Furthermore, we do not calculate the fundamental solution of the Pell equation $z^2-b\, x^2 = 1$, but use the fundamental unit of the associated number field $\QQ[\sqrt{b}]$. There are lists of these fundamental units in most computer algebra systems, such as the one we used, Sagemath~\cite{sagemath}.

So from now on let $\epsilon$ be the fundamental unit of the number field $\QQ[\sqrt{b}]$. This doesn't change the derivation of Inequality~\eqref{eq:matveev}, and we can calculate the constant $c_m$ by solving it numerically. Instead of using lower bounds for linear forms in logarithms, we apply Lemma~\ref{lem:continuedfraction} to Inequality~\eqref{eq:linearform_noDelta}. We set
\[
    c_0 = \frac{c_m(c_1^2 b + c_1)}{\log\epsilon}, \qquad \mu = \frac{\log\sqrt{b}}{\log\epsilon},
\]
to rewrite Inequality~\eqref{eq:linearform_noDelta} as
\begin{equation}\label{eq:linearform_CF}
    \abs{ (m_2n_1-m_1n_2) - (m_2-m_1)\, \mu } < c_0 \epsilon^{-2n_1}.    
\end{equation}

For $\mu$ we ensure that the logarithms are computed with a high enough precision so that the convergents are correct. Since by Lemma~\ref{lem:continuedfraction} we are looking for the first convergent greater than $q=(m_2-m_1) < c_m$, it is sufficient to take a precision some order of magnitude greater than $c_m$. Applying Lemma~\ref{lem:continuedfraction} to the first convergent
\[
    \frac{p_k}{q_k} = [a_0; a_1, \ldots, a_k]
\]
of $\mu = [a_0; a_1, \ldots]$ that satisfies $q_k > c_m$ and comparing to the upper bound of Inequality~\eqref{eq:linearform_CF} gives the bound
\[
    c_{n_1} = \left\lfloor \frac{\log q_k + \log c_0 + \log(2+\max\set{a_1, \ldots, a_k})}{2\log\epsilon} \right\rfloor
\]
for $n_1$. Note that we can expect $c_{n_1}$ to be relatively small, somewhere in the order of $\log c_m / \log\epsilon$.

Assuming that there are two solutions, so that we can formulate Inequality~\eqref{eq:linearform_noDelta}, the smaller solution must be one of the candidates
\[
    x_n = \frac{ \epsilon^n - \epsilon^{-n} }{2\sqrt{b}}, \qquad 1 \leq n \leq c_{n_1}.
\]

We can iteratively compute the candidates in the number field $\QQ[\sqrt{b}]$ quite efficiently, and loop over all of them. It may happen (e.g. for $b=2$) that not every $x_n$ is an integer, and since $\epsilon$ is the fundamental unit, not the fundamental solution, $x_n$ does not necessarily have to be a solution to $z^2-b\, x^2 = 1$. Furthermore, if $x_n \leq \sqrt{1+b} < \sqrt{1+a} $ then it cannot be a solution to $x^2 - a\, y^2 = 1$. We can skip these candidates.

For the rest, we look at the quantity $\gamma_n = x_n + \sqrt{x_n^2 - 1}$. If $x_n$ is also a solution to $x^2 - a\, y^2 = 1$, then $\gamma_n = x_n + \sqrt{a}\, y$. By computing the prime factorisation
\[
    x_n^2 - 1 = \prod_{i=1}^s p_i^{\alpha_i},
\]
we can find all $a$ and $y$ such that $x_n^2 - a\, y^2 = 1$. Furthermore, $a$ is bounded from below by the square-free part of $x_n^2-1$,
\[
    \Pi = \prod_{ \substack{i=1 \\ \alpha_i \text{ odd} } }^s p_i.
\]

If $x_n^2-1$ is square-free, then $a = x_n^2-1$ is unique and $y$ must be $1$. Since $y=1$ is minimal we can even conclude that $\gamma_n = \delta_a$ must be the fundamental solution in this case. Otherwise we have $\gamma_n = \delta_a^l$ for some positive integer $l$ which we can bound by
\[
    c_l = \left\lfloor \frac{\log\gamma_n}{2 \sqrt{\Pi}} \right\rfloor,
\]
since $\delta_a \geq 2\sqrt{a} \geq 2\sqrt{\Pi}$.

We can substitute $\delta_a$ for $\gamma_n^{1/l}$ in the second linear form of Equation~\eqref{eq:linearform} and get
\begin{equation}\label{eq:linearform_gamma}
    0 < -m_2 \log\gamma_n + ln_2 \log\epsilon - l\log\sqrt{b} < l(c_1^2b+c_1) \epsilon^{-2n_2},
\end{equation}
where all the logarithms are now explicit. Furthermore, $m_2$ is explicitly bounded by $m_2$, and $ln_2$ is explicitly bounded by
\[
    c_{n_2}(l) = \left\lfloor \frac{c_m \log\gamma_n + l\log(c_1\sqrt{b})}{\log\epsilon} \right\rfloor
\]
according to the line above Inequality~\eqref{eq:n2_bound}.

If we set
\[
    \mu = \frac{\log\epsilon}{\log\gamma_n}, \quad \tau(l) = -\frac{l \log\sqrt{b}}{\log\gamma_n}, \quad c_2(l) = \frac{l(c_1^2b+c_1)}{\log\gamma_n}, \quad c_3(l) = \frac{2 \log\epsilon }{l},
\]
then we can write Equation~\eqref{eq:linearform_gamma} as
\[  
    \abs{ ln_2 \, \mu + \tau(l) - m_2 } < c_3(l)\, e^{-c_2(l) ln_2}.
\]

We use Theorem~\ref{thm:reduction} to reduce the bound $c_{n_2}(l)$ for every possible $l = 1, \ldots, c_l$. To do this, we first compute the continued fraction expansion $[a_0; a_1, \ldots]$ of $\mu$ and iterate over its convergents $p_k/q_k = [a_0; a_1, \ldots a_k]$.

Setting $\kappa = 1/(2lc_{n_2}(l) \set{q_k \mu})$ ensures that the first condition in the theorem holds. And the $q_k$ are denominators of convergents of $\mu$, so the fractional part $\set{q_k \mu}$ quickly becomes smaller than $1/(2lc_{n_2}(l))$, or equivalently $\kappa > 1$. We iterate until the second condition in the theorem, $\kappa \set{q_k \tau(l)} > 1$, is also satisfied and then compute the reduced bound
\[
    \left\lfloor \frac{ \log(2\kappa q_k c_2(l) ) }{ l c_3(l) } \right\rfloor
\]
for $n_2$. We then take $c_{n_2}$ to be the largest of all these reduced bounds for $1 \leq l \leq c_l$.

At this point we have a range $n+1 \leq n' \leq c_{n_2}$ for the exponent of the second solution if the candidate $x_n$ gives the first, and we can calculate the corresponding candidates $x_{n'}$ accordingly. If we indeed have two solutions, then the rational
\[
    \frac{x_{n'}^2-1}{x_n^2-1} = \frac{a\,y'^2}{a\,y^2} = \br{ \frac{y'}{y} }^2
\]
must be a square. If it is, we have found a pair of solutions. Computing the prime factorisation of $x_{n'}^2 - 1$ as well allows us to find all $a$ and $y', y$-values for which
\[
    x_{n'}^2 - a\, y'^2 = 1 = x_n^2 - a\, y^2
\]
holds, and the $z$-values would be $\sqrt{1+ b\,x_{n'}^2}$ and $\sqrt{1 + b\, x_n^2}$, respectively. It is enough to know $x_n, x_{n'}$ (and the prime factorisation of $x_n^2-1$, if we do not want to calculate it a second time) to construct the entire solution.

In conclusion, we have proved the correctness of the following algorithm.

\newcounter{procedure}
\begin{algorithm}[H]\label{algorithm}

    \caption{For fixed $b$ gives all $a$ for which $x^2-a\, y^2 = 1$ and $z^2 - b\, x^2 = 1$ has two or more solutions in positive integers.}

    \SetKwInOut{Input}{Input}
    \SetKwInOut{Output}{Output}

    \Input{A (non-square) integer $b$.}
    \Output{A list of integers $(x_1, x_2)$ containing the $x$-values of pairs of solutions to $x^2-a\,y^2 =1$ and $z^2 - b\, x^2 = 1$.}

    \BlankLine
    Get the fundamental unit $\epsilon > 1$ of the number field $\QQ[\sqrt{b}]$;

    $c_1 = 1 + \dfrac{1}{4b} + \dfrac{1}{2b\epsilon}$;
    
    $c_\text{mat} = 45\cdot 4^9 \cdot e^4 (26.25 + \log(4^2 \log(4e) ) ) $;

    $c_m = $ Solution of $-c_\text{mat} 16^2 \log\epsilon \log\sqrt{b} \log(6e\log(6e)) (m+0.52) < -2m+0.39$;

    $c_0 = \dfrac{c_m(c_1^2 b + c_1)}{\log\epsilon}$;
    
    $\mu = \dfrac{\log\sqrt{b}}{\log\epsilon}$;

    $[a_0; a_1, \ldots ] = $ Continued fraction expansion of $\mu$;


        
            

    
        
        
        
        
        
        
        
        

    \setcounter{procedure}{\value{AlgoLine}}

\end{algorithm}

\begin{algorithm}[H]

    \setcounter{AlgoLine}{\value{procedure}}

    \For{$i=1, 2, \ldots$}{
        $\dfrac{p_i}{q_i} = [a_0; a_1, \ldots, a_i]$;

        $ A = \max\set{a_1, \ldots, a_i} $;
        
        \uIf{ $q_i \geq c_m$ }{
            $c_{n_1} = \left\lfloor \dfrac{ \log q_i + \log c_0  + \log(2 + A) }{2 \log\epsilon}  \right\rfloor$ ;
            
            \textbf{break};
        }
    }
    
    \For{$n=1, \ldots, c_{n_1}$}{
        $x_n = \dfrac{\epsilon^n - \epsilon^{-n}}{2\sqrt{b}}$;

        \uIf{$x_n \not\in \ZZ$ \textbf{ or } $x_n \leq \sqrt{1+b}$ \textbf{ or } $\sqrt{1+b\,x_n^2} \not\in \ZZ $}{
            \textbf{continue};
        }

        $\gamma_n = x_n + \sqrt{x_n^2 - 1}$;

        Compute prime factor decomposition $x_n^2 - 1 = p_1^{\alpha_1} \cdots p_s^{\alpha_s}$;

        \eIf{ any $\alpha_i > 1$ }{
            $\displaystyle \Pi = \prod_{ \substack{ i = 1 \\ \alpha_i \text{ odd} } }^s p_i $;

            $c_l = \left\lfloor  \dfrac{\log\gamma_n}{\log(2\sqrt{\Pi})} \right\rfloor$;
        }{
            $c_l = 1$;
        }

        $c_{n_2} = \left[  \left\lfloor \dfrac{ c_m \log\gamma_n + l \log(c_1\sqrt{b}) }{\log\epsilon} \right\rfloor \text{ for } l=1, \ldots, c_l \right]$;

        $\mu_n = \frac{\log\epsilon}{\log\gamma_n}$;

        $[a_0; a_1, \ldots] = $ Continued fraction expansion of $\mu_n$;

        \For{$l = 1, \ldots, c_l$}{
            $ \tau =  - \dfrac{l \log\sqrt{b}}{\log\gamma_n}$;

            $c_2 = \dfrac{l(c_1^2 b + c_1)}{\log\gamma_n}$;

            $c_3 = \dfrac{2\log\epsilon}{l}$;

            \For{$i=1, 2, \ldots$}{
                $\dfrac{p_i}{q_i} = [a_0; a_1, \ldots, a_i]$;

                $\kappa = \dfrac{1}{ 2l c_{n_2}(l) \set{q_i \mu_n} }$;

                \uIf{ $\kappa > 1$ \textbf{ and } $\kappa \set{q_i \tau} > 1$ }{
                    $c_{n_2}(l) = \left\lfloor \dfrac{\log(2\kappa q_i c_2)}{lc_3} \right\rfloor $;

                    \textbf{break};
                }
            }

            $c_{n_2} = \max(c_{n_2})$;

            \For{$n' = n+1, \ldots, c_{n_2}$}{
                $x_{n'} = \dfrac{\epsilon^n - \epsilon^{-n}}{2\sqrt{b}}$;
                
                \uIf{$x_{n'} \not\in \ZZ$ \textbf{ or } $x_{n'} \leq \sqrt{1+b}$ \textbf{ or } $\sqrt{1+b\,x_{n'}^2} \not\in \ZZ $}{
                    \textbf{continue};
                }
                \uIf{ $\sqrt{\dfrac{x_{n'}^2-1}{x_n^2-1}} \in \ZZ$ }{
                    $(x_n, x_{n'})$ is a pair of solutions;
                }
            }
        }
    }
    
\end{algorithm}
\newpage

\begin{remark}
    Checking all (non-square) $b$ in the range $[1, 10000]$ using an implementation of Algorithm~\ref{algorithm} in Sage took approximately $101$ hours and $24$ minutes on a standard desktop PC. A single $b$ took about $37$ seconds on average, while the longest run was around $196$ seconds.
\end{remark}

\begin{remark}
    There are a number of points worth noting in an implementation of the above procedure that are not readily apparent from the pseudocode alone.

    First, computing the powers of the fundamental unit $\epsilon = \xi + \eta \sqrt{b}$ to get the candidates for the solutions is best done in the number field $\QQ[\sqrt{b}]$.

    Solving Inequality~\eqref{eq:matveev}, or the inequality
    \[
        \begin{split}
            -c_{\text{Mat}} 16^2 \log\delta_a \log\epsilon \log\sqrt{b} \log\br{ 6e \log(6e) \br{m + \frac{\log(c_1\sqrt{b})}{\log(2\sqrt{b})} } } < \\ -2m \log\delta_a + \log\br{1 + \frac{1}{b}}
        \end{split}
    \]
    to be more precise, numerically may fail or take too long. In this case, or as a general alternative, we can use the fact that $m < c \log m$ implies $m < 2c \log c$ to compute an upper bound $c_m = 2c \log c$ (for the explicit $c$ given by the terms in the inequality).

    When considering the continued fractional expansion of $\mu = \frac{\log\sqrt{b}}{\log\epsilon}$, we must ensure that the convergents up to the one used in the calculation of $c_{n_1}$ are the correct ones. Sagemath can compute infinite continued fraction expansions (up to a certain memory limit). Since the convergents $\frac{p_k}{q_k}$ of any number always suffice $q_k \geq \Phi^k$ for the golden ratio $\Phi$, we know that we have to compute no more than the first $\left\lfloor \frac{\log c_m}{\log\Phi} \right\rfloor$ convergents to obtain $q_k \geq c_m$.
    
    If we instead compute the finite continued fraction expansion of an approximation of $\mu$, we must also take into account the approximation error we have made. This could lead to an infinite loop if the approximation error always dominates the computed lower bound from Lemma~\ref{lem:continuedfraction}.

    It is a priori possible that the condition $\kappa \set{q_i \tau} > 1$ will never be satisfied. The algorithm can be made terminable by stopping the iteration over the convergents of $\frac{\log\epsilon}{\log\gamma_n}$ if the reduced bound would already be greater than the initial bound $\left\lfloor \frac{ c_m \log\gamma_n + l \log(c_1\sqrt{b}) }{\log\epsilon} \right\rfloor$, or if the reduced bound is too large for the (last) following step to be practically feasible. In practice, it had often been the first or second $\kappa > 1$ that also satisfied $\kappa\set{q_i \tau} > 1$.
\end{remark}

\section*{Acknowledgements}
This research was funded in whole or in part by the Austrian Science Fund (FWF) [10.55776/I4406]. For open access purposes, the author has applied a CC BY public copyright license to any author accepted manuscript version arising from this submission.

%
%

\end{document}